\documentclass[12pt]{article}

\usepackage{amssymb}
\usepackage{latexsym}
\usepackage{amsmath}
\usepackage{amsthm}
\usepackage{amscd}

\newtheorem{theorem}{Theorem}[section]
\newtheorem{proposition}[theorem]{Proposition}
\newtheorem{corollary}[theorem]{Corollary}
\newtheorem{lemma}[theorem]{Lemma}
\newtheorem{lemma-definition}[theorem]{Lemma-Definition}

\theoremstyle{definition}

\renewcommand{\frak}{\mathfrak}

\newcommand{\C}{{\mathbb C}}

\newcommand{\R}{{\mathbb R}}

\newcommand{\Z}{{\mathbb Z}}

\newcommand{\A}{\mathcal{A}}
\newcommand{\op}{\operatorname}

\newcommand{\Ker}{\op{Ker}}

\newcommand{\tensor}{\otimes}

\newcommand{\vol}{\op{vol}(Y,\lambda)}
\newcommand{\E}{\mathcal{E}}
\newcommand{\F}{\mathcal{F}}
\renewcommand{\epsilon}{\varepsilon}
\newcommand{\mc}[1]{\mathcal{#1}}

\title{The asymptotics of ECH capacities}

\author{Daniel Cristofaro-Gardiner, Michael Hutchings,\\ and Vinicius Gripp Barros Ramos\footnote{The first author was partially supported by NSF grant DMS-0838703. The second and third authors were partially supported by NSF grant DMS-1105820.}}
\date{}

\begin{document}

\maketitle

\begin{abstract}
In a previous paper, the second author used embedded contact homology (ECH) of contact three-manifolds to define ``ECH capacities'' of four-dimensional symplectic manifolds. In the present paper we prove that for a four-dimensional Liouville domain with all ECH capacities finite, the asymptotics of the ECH capacities recover the symplectic volume. This follows from a more general theorem relating the volume of a contact three-manifold to the asymptotics of the amount of symplectic action needed to represent certain classes in ECH. The latter theorem was used by the first and second authors to show that every contact form on a closed three-manifold has at least two embedded Reeb orbits.
\end{abstract}

\section{Introduction}

Define a four-dimensional {\em Liouville domain\/}\footnote{This definition of ``Liouville domain'' is slightly weaker than the usual definition, which would require that $\omega$ have a primitive $\lambda$ on $X$ which restricts to a contact form on $Y$.} to be a compact symplectic four-manifold $(X,\omega)$ with oriented boundary $Y$ such that $\omega$ is exact on $X$, and there exists a contact form $\lambda$ on $Y$ with $d\lambda=\omega|_Y$. In \cite{qech}, a sequence of real numbers
\[
0=c_0(X,\omega) < c_1(X,\omega) \le c_2(X,\omega) \le \cdots \le \infty
\]
called {\em ECH capacities\/} was defined. The definition is reviewed below in \S\ref{sec:capacities}. The ECH capacities obstruct symplectic embeddings: If $(X,\omega)$ symplectically embeds into $(X',\omega')$, then
\begin{equation}
\label{eqn:volumeconstraint}
c_k(X,\omega)\le c_k(X',\omega')
\end{equation}
for all $k$. For example, a theorem of McDuff \cite{mcd}, see also the survey \cite{pnas}, shows that ECH capacities give a sharp obstruction to symplectically embedding one four-dimensional ellipsoid into another.

The first goal of this paper is to prove the following theorem, relating the asymptotics of the ECH capacities to volume. 
This result was conjectured in \cite{qech} based on experimental evidence; it was proved in \cite[\S8]{qech} for star-shaped domains in $\R^4$ and some other examples.

\begin{theorem}
\label{thm:vol4}
\cite[Conj. 1.12]{qech}
Let $(X,\omega)$ be a four-dimensional Liouville domain such that $c_k(X,\omega)<\infty$ for all $k$. Then
\[
\lim_{k\to\infty}\frac{c_k(X,\omega)^2}{k} = 4\op{vol}(X,\omega).
\]
\end{theorem}
\noindent
Here the symplectic volume is defined by
\[
\op{vol}(X,\omega) = \frac{1}{2}\int_X\omega\wedge\omega.
\]

In particular, when all ECH capacities are finite, the embedding obstruction \eqref{eqn:volumeconstraint} for large $k$ recovers the obvious volume constraint $\op{vol}(X,\omega) \le \op{vol}(X',\omega')$. As we review below, the hypothesis that $c_k(X,\omega)<\infty$ for all $k$ is a purely topological condition on the contact structure on the boundary; for example it holds whenever $\partial X$ is diffeomorphic to $S^3$.

We will obtain Theorem~\ref{thm:vol4} as a corollary of the more general Theorem~\ref{thm:main} below, which also has applications to refinements of the Weinstein conjecture in Corollary~\ref{cor:two}. To state Theorem~\ref{thm:main}, we first need to review some notions from embedded contact homology (ECH).
More details about ECH may be found in \cite{bn} and the references therein.

\subsection{Embedded contact homology}
\label{sec:ECH}

Let $Y$ be a closed oriented three-manifold and let $\lambda$ be a contact form on $Y$, meaning that $\lambda\wedge d\lambda>0$. The contact form $\lambda$ determines a contact structure $\xi=\Ker(\lambda)$, and the Reeb vector field $R$ characterized by $d\lambda(R,\cdot)=0$ and $\lambda(R)=1$. Assume that $\lambda$ is nondegenerate, meaning that all Reeb orbits are nondegenerate. Fix $\Gamma\in H_1(Y)$.  The {\em embedded contact homology\/} $ECH(Y,\xi,\Gamma)$ is the homology of a chain complex $ECC(Y,\lambda,\Gamma,J)$ over $\Z/2$ defined as follows.

A generator of the chain complex is a finite set of pairs
$\alpha=\{(\alpha_i,m_i)\}$ where the $\alpha_i$ are distinct embedded
Reeb orbits, the $m_i$ are positive integers, $m_i=1$ whenever
$\alpha_i$ is hyperbolic, and the total homology class
$\sum_im_i[\alpha_i]=\Gamma\in H_1(Y)$. To define the chain complex
differential $\partial$ one chooses a generic almost complex structure
$J$ on $\R\times Y$ such that $J(\partial_s)=R$ where $s$ denotes the
$\R$ coordinate, $J(\xi)=\xi$ with $d\lambda(v,Jv)\ge 0$ for
$v\in\xi$, and $J$ is $\R$-invariant. Given another chain complex
generator $\beta=\{(\beta_j,n_j)\}$, the differential coefficient
$\langle\partial\alpha,\beta\rangle\in\Z/2$ is a mod 2 count of
$J$-holomorphic curves in $\R\times Y$ that converge as currents to
$\sum_im_i\alpha_i$ as $s\to+\infty$ and to $\sum_jn_j\beta_j$ as
$s\to-\infty$, and that have ``ECH index'' equal to $1$. The
definition of the ECH index is explained in \cite{ir}; all we need to
know here is that the ECH index defines a relative $\Z/d$-grading on the chain
complex, where $d$ denotes the divisibility of
$c_1(\xi)+2\op{PD}(\Gamma)$ in $H^2(Y;\Z)$ mod torsion.  It is shown
in \cite[\S7]{obg1} that $\partial^2=0$.

One now defines $ECH(Y,\lambda,\Gamma,J)$ to be the homology of the chain complex $ECC(Y,\lambda,\Gamma,J)$. Taubes \cite{e1} proved that if $Y$ is connected, then there is a canonical isomorphism of relatively graded $\Z/2$-modules
\begin{equation}
\label{eqn:taubes}
ECH_*(Y,\lambda,\Gamma,J) = \widehat{HM}^{-*}(Y,\frak{s}_\xi+\op{PD}(\Gamma)).
\end{equation}
Here $\widehat{HM}^{-*}$ denotes the `from' version of Seiberg-Witten Floer cohomology as defined by Kronheimer-Mrowka \cite{km}, with $\Z/2$ coefficients\footnote{One can define ECH with integer coefficients \cite[\S9]{obg2}, and the isomorphism \eqref{eqn:taubes} also exists over $\Z$, as shown in \cite{e3}. However $\Z/2$ coefficients will suffice for this paper.}, and with the sign of the relative grading reversed. Also, $\frak{s}_\xi$ denotes the spin-c structure determined by the oriented 2-plane field $\xi$, see e.g. \cite[Ex. 8.2]{tw}. It follows that, whether or not $Y$ is connected, $ECH(Y,\lambda,\Gamma,J)$ depends only on $Y$, $\xi$, and $\Gamma$, and so can be denoted by $ECH_*(Y,\xi,\Gamma)$.

There is a degree $-2$ map
\begin{equation}
\label{eqn:Umap}
U:ECH_*(Y,\xi,\Gamma) \longrightarrow ECH_{*-2}(Y,\xi,\Gamma).
\end{equation}
This map on homology is induced by a chain map which counts
$J$-holomorphic curves with ECH index $2$ that pass through a base
point in $\R\times Y$. When $Y$ is connected, the $U$ map
\eqref{eqn:Umap} does not depend on the choice of base point, and
agrees under Taubes's isomorphism \eqref{eqn:taubes} with an analogous
map on Seiberg-Witten Floer cohomology \cite{e5}. If $Y$ is
disconnected, then there is one $U$ map for each component of $Y$.

Although ECH is a topological invariant by \eqref{eqn:taubes}, it contains a distinguished class which can distinguish some contact structures. Namely, the empty set of Reeb orbits is a generator of $ECC(Y,\lambda,0,J)$; it is a cycle by the conditions on $J$, and so it defines a distinguished class
\begin{equation}
\label{eqn:ci}
[\emptyset]\in ECH(Y,\xi,0),
\end{equation}
called the {\em ECH contact invariant\/}.
Under the isomorphism \eqref{eqn:taubes}, the ECH contact invariant agrees with an analogous contact invariant in Seiberg-Witten Floer cohomology \cite{e5}.

There is also a ``filtered'' version of ECH, which is sensitive to the contact form and not just the contact structure. If $\alpha=\{(\alpha_i,m_i)\}$ is a generator of the chain complex $ECC(Y,\lambda,\Gamma,J)$, its {\em symplectic action\/} is defined by
\begin{equation}
\label{eqn:symplecticaction}
\mc{A}(\alpha) = \sum_im_i\int_{\alpha_i}\lambda.
\end{equation}
It follows from the conditions on the almost complex structure $J$ that if the differential coefficient $\langle\partial\alpha,\beta\rangle\neq 0$ then $\mc{A}(\alpha)>\mc{A}(\beta)$. Consequently, for each $L\in\R$, the span of those generators $\alpha$ with $\mc{A}(\alpha)<L$ is a subcomplex, which is denoted by $ECC^L(Y,\lambda,\Gamma,J)$. The homology of this subcomplex is the {\em filtered ECH\/} which is denoted by $ECH^L(Y,\lambda,\Gamma)$.
Inclusion of chain complexes induces a map
\begin{equation}
\label{eqn:iim}
ECH^L(Y,\lambda,\Gamma)\longrightarrow ECH(Y,\xi,\Gamma).
\end{equation}
It is shown in \cite[Thm.\ 1.3]{cc2} that $ECH^L(Y,\lambda,\Gamma)$ and
the map \eqref{eqn:iim} do not depend on the almost complex structure $J$.

A useful way to extract invariants of the contact form out of filtered ECH is as follows. Given a nonzero class $\sigma\in ECH(Y,\xi,\Gamma)$, define
\[
c_\sigma(Y,\lambda)\in\R
\]
to be the infimum over $L$ such that the class $\sigma$ is in the
image of the inclusion-induced map \eqref{eqn:iim}.
So far we have
been assuming that the contact form $\lambda$ is nondegenerate. If
$\lambda$ is degenerate, one defines
$c_\sigma(Y,\lambda)=\lim_{n\to\infty}c_\sigma(Y,\lambda_n)$, where
$\{\lambda_n\}$ is a sequence of nondegenerate contact forms which
$C^0$-converges to $\lambda$, cf.\ \cite[\S3.1]{qech}.

\subsection{ECH capacities}
\label{sec:capacities}

Let $(Y,\lambda)$ be a closed contact three-manifold and assume that the ECH contact invariant \eqref{eqn:ci} is nonzero. Given a nonnegative integer $k$, define $c_k(Y,\lambda)$ to be the minimum of $c_\sigma(Y,\lambda)$, where $\sigma$ ranges over classes in $ECH(Y,\xi,0)$ such that $A\sigma=[\emptyset]$ whenever $A$ is a composition of $k$ of the $U$ maps associated to the components of $Y$. If no such class $\sigma$ exists, define $c_k(Y,\lambda)=\infty$. The sequence $\{c_k(Y,\lambda)\}_{k=0,1,\ldots}$ is called the {\em ECH spectrum\/} of $(Y,\lambda)$.

Now let $(X,\omega)$ be a Liouville domain with boundary $Y$ and let $\lambda$ be a contact form on $Y$ with $d\lambda=\omega|_Y$. One then defines the ECH capacities of $(X,\omega)$ in terms of the ECH spectrum of $(Y,\lambda)$ by
\[
c_k(X,\omega) = c_k(Y,\lambda).
\]
This definition is valid because the ECH contact invariant of $(Y,\lambda)$ is nonzero by \cite[Thm.\ 1.9]{cc2}. It follows from \cite[Lem. 3.9]{qech} that $c_k(X,\omega)$ does not depend on the choice of contact form $\lambda$ on $Y$ with $d\lambda=\omega|_Y$.

Note the volume of a Liouville domain as above satisfies
\begin{equation}
\label{eqn:liouvillevolume}
\op{vol}(X,\omega) = \frac{1}{2}\op{vol}(Y,\lambda),
\end{equation}
where the volume of a contact three-manifold is defined by
\begin{equation}
\label{eqn:contactvolume}
\op{vol}(Y,\lambda) = \int_Y \lambda\wedge d\lambda.
\end{equation}
To prove \eqref{eqn:liouvillevolume}, let $\lambda'$ be a primitive of $\omega$ on $X$, and use Stokes's theorem on $X$ and then again on $Y$ to obtain
\[
\int_X\omega\wedge\omega = \int_Y \lambda'\wedge \omega = \int_Y \omega\wedge \lambda = \int_Y d\lambda\wedge \lambda.
\]

By equation \eqref{eqn:liouvillevolume}, Theorem~\ref{thm:vol4} is now a consequence of the following result about the ECH spectrum:

\begin{theorem}
\label{thm:vol2}
\cite[Conj.\ 8.1]{qech}
Let $(Y,\lambda)$ be a closed contact three-manifold with nonzero ECH contact invariant. If $c_k(Y,\lambda)<\infty$ for all $k$, then
\[
\lim_{k\to\infty}\frac{c_k(Y,\lambda)^2}{k} = 2\op{vol}(Y,\lambda).
\]
\end{theorem}
\noindent

Note that the hypothesis $c_k(Y,\lambda)<\infty$ just means that the ECH contact invariant is in the image of all powers of the $U$ map when $Y$ is connected, or all compositions of powers of the $U$ maps when $Y$ is disconnected. The comparison with Seiberg-Witten theory implies that this is possible only if  $c_1(\xi)\in H^2(Y;\Z)$ is torsion; see \cite[Rem.\ 4.4(b)]{qech}.

By \cite[Prop.\ 8.4]{qech}, to prove Theorem~\ref{thm:vol2} it suffices to consider the case when $Y$ is connected. Theorem~\ref{thm:vol2} in this case follows from our main theorem which we now state.

\subsection{The main theorem}

Recall from \S\ref{sec:ECH} that if $c_1(\xi)+2\op{PD}(\Gamma)\in H^2(Y;\Z)$ is torsion, then $ECH(Y,\xi,\Gamma)$ has a relative $\Z$-grading, and we can arbitrarily refine this to an absolute $\Z$-grading. The main theorem is now:

\begin{theorem}
\label{thm:main}
\cite[Conj. 8.7]{qech}
Let $Y$ be a closed connected contact three-manifold with a contact form $\lambda$ and let $\Gamma\in H_1(Y)$. Suppose that $c_1(\xi)+2\op{PD}(\Gamma)$ is torsion in $H^2(Y;\Z)$, and let $I$ be an absolute $\Z$-grading of $ECH(Y,\xi,\Gamma)$. Let $\{\sigma_k\}_{k\ge1}$ be a sequence of nonzero homogeneous classes in $ECH(Y,\xi,\Gamma)$ with $\lim_{k\to\infty} I(\sigma_k)=\infty$. Then
\begin{equation}
\label{eqn:main}
\lim_{k\to\infty} \frac{c_{\sigma_k}(Y,\lambda)^2}{I(\sigma_k)} = \op{vol}(Y,\lambda).
\end{equation}
\end{theorem}

The following application of Theorem~\ref{thm:main} was obtained in \cite{ch}:

\begin{corollary}
\label{cor:two}
\cite[Thm. 1.1]{ch}
Every (possibly degenerate) contact form on a closed three-manifold has at least two embedded Reeb orbits.
\end{corollary}

The proof of Theorem~\ref{thm:main} has two parts. In \S\ref{sec:ub} we show that the left hand side of \eqref{eqn:main} (with lim replaced by lim sup) is less than or equal to the right hand side. This is actually all that is needed for Corollary~\ref{cor:two}. In \S\ref{sec:lb} we show that the left hand side (with lim replaced by lim inf) is greater than or equal to the right hand side. The two arguments are independent of each other and can be read in either order. The proof of the upper bound uses ingredients from Taubes's proof of the isomorphism \eqref{eqn:taubes}. The proof of the lower bound uses properties of ECH cobordism maps to reduce to the case of a sphere, where \eqref{eqn:main} can be checked explicitly.

Both arguments use Seiberg-Witten theory; in particular, Seiberg-Witten theory is used to define ECH cobordism maps in \cite{cc2}. However the proof of the lower bound given here would not need Seiberg-Witten theory if one could give an alternate construction of ECH cobordism maps.

\section{The upper bound}
\label{sec:ub}

In this section we prove the upper bound half of Theorem~\ref{thm:main}:

\begin{proposition}\label{prop:up}
Under the assumptions of Theorem~\ref{thm:main},
\begin{equation}
\label{eqn:up}
\limsup_{k\to\infty} \frac{c_{\sigma_k}(Y,\lambda)^2}{I(\sigma_k)}\le\vol.
\end{equation}
\end{proposition}

To prove Proposition~\ref{prop:up}, we can assume without loss of generality that $\lambda$ is nondegenerate. To see this, assume that \eqref{eqn:up} holds for nondegenerate contact forms and suppose that $\lambda$ is degenerate. We can find a sequence of functions $f_1>f_2>\cdots > 1$, which $C^0$-converges to $1$, such that $f_n\lambda$ is nondegenerate for each $n$. It follows from the monotonicity property in \cite[Lem.\ 4.2]{qech} that
\[
c_{\sigma_k}(Y,\lambda)
\le
c_{\sigma_k}(Y,f_n\lambda)
\]
for every $n$ and $k$. For each $n$, it follows from this and the inequality \eqref{eqn:up} for $\lambda_n$ that
\[
\limsup_{k\to\infty} \frac{c_{\sigma_k}(Y,\lambda)^2}{I(\sigma_k)}\le \op{vol}(Y,f_n\lambda).
\]
Since $\lim_{n\to\infty}\op{vol}(Y,f_n\lambda) = \op{vol}(Y,\lambda)$, we deduce the inequality \eqref{eqn:up} for $\lambda$.

Assume henceforth that $\lambda$ is nondegenerate.  In \S\ref{sec:swf}--\S\ref{sec:mm} below we review some aspects of Taubes's proof of the isomorphism \eqref{eqn:taubes} and prove some related lemmas. In \S\ref{sec:proofup} we use these to prove Proposition~\ref{prop:up}.

\subsection{Seiberg-Witten Floer cohomology}
\label{sec:swf}

The proof of the isomorphism \eqref{eqn:taubes} involves perturbing the Seiberg-Witten equations on $Y$. To write down the Seiberg-Witten equations we first need to choose a Riemannian metric on $Y$.
Let $J$ be a generic almost complex structure on $\R\times Y$ as needed to define the ECH chain complex.  The almost complex structure $J$ determines a Riemannian metric $g$ on $Y$ such that the Reeb vector field $R$ has length $1$ and is orthogonal to the contact planes $\xi$, and
\begin{equation}
\label{eqn:Jg}
g(v,w) = \frac{1}{2}d\lambda(v,Jw), \quad\quad v,w\in\xi_y.
\end{equation}
Note that this metric satisfies
\begin{equation}
\label{eqn:metric}
|\lambda| = 1, \quad\quad {*}d\lambda = 2\lambda.
\end{equation}
One could dispense with the factors of $2$ in \eqref{eqn:Jg} and \eqref{eqn:metric}, but we are keeping them for consistency with \cite{tw} and its sequels.

Let $\mathbb{S}$ denote the spin bundle for the spin-c structure 
$\mathfrak{s}_{\xi}+\op{PD}(\Gamma)$. The inputs to the Seiberg-Witten equations for this spin-c structure are a connection ${\mathbb A}$ on $\det(\mathbb{S})$ and a section $\psi$ of $\mathbb{S}$.
The spin bundle $\mathbb{S}$ splits as a direct sum
\[
\mathbb{S}=E\oplus (E\otimes \xi),
\]
where $E$ and $E\otimes \xi$ are, respectively, the $+i$ and $-i$ eigenspaces of Clifford multiplication by $\lambda$. Here $\xi$ is regarded as a complex line bundle using the metric and the orientation.  A connection ${\mathbb A}$ on $\op{det}(\mathbb{S})$ is then equivalent to a (Hermitian) connection $A$ on $E$ via the relation ${\mathbb A} = A_0+2A$, where $A_0$ is a distinguished connection on $\xi$ reviewed in \cite[\S 2.1]{tw1}.

For a positive real number $r$, consider the following version of the perturbed Seiberg-Witten equations for a connection $A$ on $E$ and spinor $\psi$:  
\begin{equation}
\label{eqn:taubessw}
\begin{split}
{*}F_A&=r(\langle cl(\cdot)\psi,\psi\rangle-i\lambda)+i({*}d\mu+\bar{\omega})\\
D_A\psi&= 0.
\end{split}
\end{equation}
Here $F_A$ denotes the curvature of $A$; $cl$ denotes Clifford multiplication; $\bar{\omega}$ denotes the harmonic 1-form such that ${*} \bar{\omega}/\pi$ represents $c_1(\xi)\in H^2(Y;\R)$; and $\mu$ is a generic $1$-form such that $d\mu$ has ``P-norm" less than $1$, see \cite[\S2.1]{tw1}. Finally, $D_A$ denotes the Dirac operator determined by the connection ${\mathbb A}$ on $\op{det}({\mathbb S})$ corresponding to the connection $A$ on $E$.

The group of {\em gauge transformations\/} $C^\infty(Y,S^1)$ acts on the space of pairs $({\mathbb A},\psi)$ by $g\cdot({\mathbb A},\psi)=({\mathbb A}-2g^{-1}dg,g\psi)$. The quotient of the space of pairs $({\mathbb A},\psi)$ by the group of gauge transformations is called the {\em configuration space\/}. The set of solutions to \eqref{eqn:taubessw} is invariant under gauge transformations.  A solution to the Seiberg-Witten equations is called {\em reducible} if $\psi \equiv 0$ and {\em irreducible} otherwise.
An irreducible solution is called {\em nondegenerate} if it is cut out transversely after modding out by gauge transformations, see \cite[\S 3.1]{tw1}.

For fixed $\mu$, when $r$ is not in a certain discrete set, there are only finitely many irreducible solutions to \eqref{eqn:taubessw} and these are all nondegenerate.  In this case one can define the Seiberg-Witten Floer cohomology chain complex with $\Z/2$ coefficients, which we denote by $\widehat{CM}^*(Y,\frak{s}_{\xi,\Gamma},\lambda,J,r)$. The chain complex is generated by irreducible solutions to \eqref{eqn:taubessw}, along with additional generators determined by the reducible solutions. The differential counts solutions to a small abstract perturbation of the four-dimensional Seiberg-Witten equations on $\R\times Y$. In principle the chain complex differential may depend on the choice of abstract perturbation, but since the abstract perturbation is irrelevant to the proof of Proposition~\ref{prop:up}, we will omit it from the notation.

\subsection{The grading}
\label{sec:grading}

Under our assumption that $c_1(\xi)+2\op{PD}(\Gamma)$ is torsion, the chain complex $\widehat{CM}^*$ has a noncanonical absolute $\Z$-grading defined as follows, cf.\ \cite[\S14.4]{km}. The linearization of the equations \eqref{eqn:taubessw} modulo gauge equivalence at a pair $(A,\psi)$, not necessarily solving the equations \eqref{eqn:taubessw}, defines a self-adjoint Fredholm operator ${\mathcal L}_{A,\psi}$.  If $(A,\psi)$ is a nondegenerate irreducible solution to \eqref{eqn:taubessw}, then the operator ${\mathcal L}_{A,\psi}$ has trivial kernel, and one defines the grading $gr(A,\psi)\in\Z$ to be the spectral flow from ${\mathcal L}_{A,\psi}$ to a reference self-adjoint Fredholm operator ${\mathcal L}_0$ between the same spaces with trivial kernel. The grading function $gr$ depends on the choice of reference operator; fix one below. To describe the gradings of the remaining generators, recall that the set of reducible solutions modulo gauge equivalence is a torus ${\mathbb T}$ of dimension $b_1(Y)$. As explained in \cite[\S35.1]{km}, one can perturb the Seiberg-Witten equations using a Morse function
\begin{equation}
\label{eqn:ft}
f:{\mathbb T}\to\R,
\end{equation}
so that the chain complex generators arising from reducibles are identified with pairs $((A,0),\phi)$, where $(A,0)$ is a critical point of $f$ and $\phi$ is a suitable eigenfunction of the Dirac operator $D_A$.  The grading of each such generator is less than or equal to $gr(A,0)$, where the latter is defined as the spectral flow to ${\mathcal L}_0$ from an appropriate perturbation of the operator ${\mathcal L}_{A,0}$.

We will need the following key result of Taubes relating the grading to the Chern-Simons functional.
Fix a reference connection $A_E$ on $E$. Given any other connection $A$ on $E$, define the {\em Chern-Simons functional\/}
\begin{equation}
\label{eqn:cs}
cs(A) = -\int_Y(A-A_E)\wedge(F_A+F_{A_E}-2i{*}\bar{\omega}).
\end{equation}
Note that this functional is gauge invariant because the spin-c structure $\frak{s}_\xi + \op{PD}(\Gamma)$ is assumed torsion.

\begin{proposition}
\label{prop:sfcs}
\cite[Prop.\ 5.1]{tw1}
There exists $K,r_* > 0$ such that for all $r>r_*$, if $(A,\psi)$ is a nondegenerate irreducible solution to \eqref{eqn:taubessw}, or a reducible solution which is a critical point of \eqref{eqn:ft}, then
\begin{equation}
\label{eqn:speceq}
\left|gr(A,\psi)+\frac{1}{4\pi^2}cs(A)\right| < K r^{31/16}.
\end{equation}
\end{proposition}

\subsection{Energy}
\label{sec:energy}

Given a connection $A$ on $E$, define the {\em energy\/}
\[
\E(A) = i\int_Y \lambda\wedge F_A.
\]
Filtered ECH has a Seiberg-Witten analogue defined using the energy functional as follows.  Fix $r$ such that the chain complex $\widehat{CM}^*$ is defined. Given a real number $L$, define $\widehat{CM}^*_L$ to be the submodule of $\widehat{CM}^*$ spanned by generators with energy less than $2\pi L$.  It is shown in \cite{tw1}, as reviewed in \cite[Lem.\ 2.3]{cc2}, that if $r$ is sufficiently large, then all chain complex generators with energy less than $2\pi L$ are irreducible, and $\widehat{CM}^*_L$ is a subcomplex, whose homology we denote by $\widehat{HM}^*_L$.  Moreover, as shown in \cite{tw1} and reviewed in \cite[Eq.\ (35)]{cc2}, if there are no ECH generators of action exactly $L$ and if $r$ is sufficiently large, then there is a canonical isomorphism of relatively graded chain complexes
\begin{equation}
\label{eqn:cci}	
ECC_*^L(Y,\lambda,\Gamma,J)
\longrightarrow\widehat{CM}^{-*}_L(Y,\frak{s}_{\xi,\Gamma},\lambda_1,J_1,r).
\end{equation}
Here $(\lambda_1,J_1)$ is an ``L-flat approximation'' to $(\lambda,J)$, which is obtained by suitably modifying $(\lambda,J)$ near the Reeb orbits of action less than $L$; the precise definition is reviewed in \cite[\S3.1]{cc2} and will not be needed here.

The isomorphism \eqref{eqn:cci} is induced by a bijection on generators; the idea is that in the $L$-flat case\footnote{In the non-$L$-flat case, there may be several Seiberg-Witten solutions corresponding to the same ECH generator, and/or Seiberg-Witten solutions corresponding to sets of Reeb orbits with multiplicities which are not ECH generators because they include hyperbolic orbits with multiplicity greater than one. See \cite[\S5.c, Part 2]{e1}.}, if $r$ is sufficiently large, then for every ECH generator $\alpha$ of action less than $L$, there is a corresponding irreducible solution $(A,\psi)$ to \eqref{eqn:taubessw} such that the zero set of the $E$ component of $\psi$ is close to the Reeb orbits in $\alpha$, the curvature $F_A$ is concentrated near these Reeb orbits, and the energy of this solution is approximately $2\pi\A(\alpha)$.

The isomorphism of chain complexes \eqref{eqn:cci} induces an isomorphism on homology
\begin{equation}
\label{eqn:Lflat}
ECH_*^L(Y,\lambda,\Gamma,J) \stackrel{\simeq}{\longrightarrow}  \widehat{HM}^{-*}_L(Y,\frak{s}_{\xi,\Gamma},\lambda_1,J_1,r),
\end{equation}
and inclusion of chain complexes defines a map
\begin{equation}
\label{eqn:iim1}
\widehat{HM}^{-*}_L(Y,\frak{s}_{\xi,\Gamma},\lambda_1,J_1,r) \longrightarrow \widehat{HM}^{-*}(Y,\mathfrak{s}_{\xi,\Gamma}).
\end{equation}
Composing the above two maps gives a map 
\begin{equation}
\label{eqn:TL}
ECH_*^L(Y,\lambda,\Gamma,J) \longrightarrow \widehat{HM}^{-*}(Y,\mathfrak{s}_{\xi,\Gamma}).
\end{equation}
The isomorphism \eqref{eqn:taubes} is the direct limit over $L$ of the maps \eqref{eqn:TL}.

\subsection{Volume in Seiberg-Witten theory}
\label{sec:volume}

The volume enters into the proof of Proposition~\ref{prop:up} in two essential ways.

The first way is as follows. It is shown in \cite[\S3]{tw} that for any given grading, there are no generators arising from reducibles if $r$ is sufficiently large. That is, given an integer $j$, let $s_j$ be the supremum of all values of $r$ such that there exists a generator of the chain complex $\widehat{CM}^*(Y,\frak{s}_{\xi,\Gamma},\lambda,J,r)$ with grading at least $-j$ associated to a reducible solution to \eqref{eqn:taubessw}. Then $s_j<\infty$ for all $j$.

We now give an upper bound on the number $s_j$ in terms of the volume.
Fix $0<\delta<\frac{1}{16}$. Given a positive integer $j$, let $r_j$ be the largest real number such that
\begin{equation}
\label{eqn:rj}
j = \frac{1}{16\pi^2}r_j^2\op{vol}(Y,\lambda) - r_j^{2-\delta}.
\end{equation}

\begin{lemma}
\label{lem:red}
If $j$ is sufficiently large, then $s_j < r_j$.
\end{lemma}

\begin{proof}
Let $A_0$ be a connection on $E$ with $F_{A_0} = id\mu + i{*}\overline{\omega}$.
Observe that $(A_r^{red},\psi)=(A_0-\frac{1}{2}ir\lambda,0)$ is a solution to \eqref{eqn:taubessw}.  Moreover, every other reducible solution is given by $(A,0)$, where $A=A_r^{red}+\alpha$ for a closed $1$-form $\alpha$.  It follows from \eqref{eqn:cs} that
\begin{equation}
\label{eqn:keycomputation1}
cs(A) = cs(A_r^{red})
\end{equation}
and
\begin{equation}
\label{eqn:keycomputation2}
cs(A_r^{red})=\frac{1}{4}r^2\vol+O(r).
\end{equation}
Now suppose that $j$ is sufficiently large that $r_j>r_*$ where $r_*$ is the constant in Proposition~\ref{prop:sfcs}. Suppose that $r>r_j$ and that $(A,0)$ is a chain complex generator with $gr(A,0)\ge -j$.
Then equation \eqref{eqn:rj} implies that
\[
gr(A,0) \ge -\frac{1}{16\pi^2}r^2\op{vol}(Y,\lambda) + r^{2-\delta}.
\]
Combining this with \eqref{eqn:keycomputation1} and \eqref{eqn:keycomputation2} gives
\[
gr(A,0) + \frac{1}{4\pi^2}cs(A) \ge r^{2-\delta} + O(r).
\]
This contradicts Proposition~\ref{prop:sfcs} if $r$ is sufficiently large, which is the case if $j$ is sufficiently large.
\end{proof}

The second essential way that volume enters into the proof of Proposition~\ref{prop:up} is via the following a priori upper bound on the energy:

\begin{lemma}
\label{lem:energybound}
There is an $r$-independent constant $C$ such that any solution $(A,\psi)$ to \eqref{eqn:taubessw} satisfies
\begin{equation}
\label{eqn:energybound}
\E(A) \le \frac{r}{2}\op{vol}(Y,\lambda)+C.
\end{equation}
\end{lemma}

\begin{proof}
This follows from \cite[Eq.\ (2.7)]{tw1}, which is proved using a priori estimates on solutions to the Seiberg-Witten equations. Note that there is a factor of $1/2$ in \eqref{eqn:energybound} which is not present in \cite[Eq.\ (2.7)]{tw1}. The reason is that the latter uses the Riemannian volume as defined by the metric \eqref{eqn:metric}, which is half of the contact volume 
\eqref{eqn:contactvolume} which we are using.
\end{proof}

\subsection{Max-min families}
\label{sec:maxmin}

Given a connection $A$ on $E$ and a section $\psi$ of ${\mathbb S}$, define a functional
\[
\F(A,\psi) = \frac{1}{2}(cs(A)-r\E(A))+e_\mu(A)+ r\int_Y \langle D_A\psi,\psi\rangle d\text{vol},
\]
where
\[
e_\mu(A) = i\int_Y F_A\wedge\mu.
\]
Since the spin-c structure $\frak{s}_\xi+\op{PD}(\Gamma)$ is assumed torsion, the functional ${\mathcal F}$ is gauge invariant.
The significance of the functional $\mathcal{F}$ is that the differential on the chain complex $\widehat{CM}^*$ counts solutions to abstract perturbations of the upward gradient flow equation for $\mathcal{F}$. In particular, $\mathcal{F}$ agrees with an appropriately perturbed version of the Chern-Simons-Dirac functional from \cite{km}, up to addition of an $r$-dependent constant, see \cite[Eq. (98)]{cc2}.

A key step in Taubes's proof of the Weinstein conjecture \cite{tw1} is to use a ``max-min" approach to find a sequence $(r_n,\psi_n,A_n)$, where  $r_n\to\infty$ and $(\psi_n,A_n)$ is a solution to \eqref{eqn:taubessw} for $r=r_n$ with an $n$-independent bound on the energy.  We will use a similar construction in the proof of Proposition~\ref{prop:up}.  

Specifically, fix an integer $j$, and let $s_j$ be the number from \S\ref{sec:volume}.  Let $\hat{\sigma}\in\widehat{HM}^*(Y,\frak{s}_{\xi,\Gamma})$ be a nonzero homogeneous class with grading greater than or equal to $-j$.  Fix $r > s_j$ for which the chain complex $\widehat{CM}^*(Y,\frak{s}_{\xi,\Gamma},\lambda,J,r)$ is defined. Since we are using $\Z/2$-coefficients, any cycle representing the class $\hat{\sigma}$ has the form $\eta = \Sigma_i (A_i,\psi_i)$, where the pairs $(A_i,\psi_i)$ are distinct gauge equivalence classes of solutions to \eqref{eqn:taubessw}.   Define ${\F}_{\min}(\eta)=\min_i\F(A_i,\psi_i)$, and
\[
\F_{\hat{\sigma}}(r)=\max_{[\eta]=\hat{\sigma}}\F_{\min}(\eta).
\]
 Note that it is natural to take the max-min here because the differential on the chain complex increases $\F$; compare \cite[Prop.\ 10.7]{tw}. Note also that $\F_{\hat{\sigma}}(r)$ must be finite because there are only finitely many irreducible solutions to $\eqref{eqn:taubessw}$.

The construction in \cite[\S4.e]{e1} shows that for any such class $\hat{\sigma}$, there exists a smooth family of solutions $(A_{\hat{\sigma}}(r),\psi_{\hat{\sigma}}(r))$ to \eqref{eqn:taubessw} with the same grading as $\hat{\sigma}$, defined for each $r>s_j$ for which the chain complex $\widehat{CM}^*$ is defined, such that $\F_{\hat{\sigma}}(r)=\F(A_{\hat{\sigma}}(r),\psi_{\hat{\sigma}}(r))$. 
This family is smooth for each interval on which it is defined, but may not extend continuously over those values of $r$ for which the chain complex $\widehat{CM}^*$ is not defined.
 Call the family $(A_{\hat{\sigma}}(r),\psi_{\hat{\sigma}}(r))_{r>s_j}$ a {\em max-min family} for $\hat{\sigma}$.  Given such a max-min family, if $r>s_j$ is such that $\widehat{CM}^*$ is defined, define $\E_{\hat{\sigma}}(r)=\E(A_{\hat{\sigma}}(r),\psi_{\hat{\sigma}}(r))$. 

\begin{lemma}
\label{lem:fam}
\begin{itemize}
 \item[(a)] $\F_{\hat{\sigma}}(r)$ extends to a continuous and piecewise smooth function of $r\in(s_j,\infty)$.
\item[(b)] $\dfrac{d}{dr}\F_{\hat{\sigma}}(r) = -\dfrac{1}{2}\E_{\hat{\sigma}}(r)$ for all $r>s_j$ such that $\widehat{CM}^*$ is defined. 
\end{itemize}
\begin{proof}
$(a)$ follows from \cite[Prop.\ 4.7]{e1}; and $(b)$ follows from \cite[Eq.\ (4.6)]{tw1}, see also \cite[Lem.\ 10.8]{tw}. 
\end{proof}    
\end{lemma}
\noindent
In particular, $\E_{\hat{\sigma}}(r)$ does not depend on the choice of max-min family.

\subsection{Max-min energy and min-max symplectic action}
\label{sec:mm}

The numbers $\E_{\hat{\sigma}}(r)$ from \S\ref{sec:maxmin} are related to the numbers $c_\sigma(Y,\lambda)$ from \S\ref{sec:capacities} as follows:

\begin{proposition}
\label{prop:energyaction}  
Let $\sigma$ be a nonzero homogeneous class in $ECH(Y,\xi,\Gamma)$, and let $\hat{\sigma}\in \widehat{HM}^*(Y,\frak{s}_{\xi,\Gamma})$ denote the class corresponding to $\sigma$ under the isomorphism \eqref{eqn:taubes}.  Then
\[
\lim_{r \to \infty} \E_{\hat{\sigma}}(r) = 2\pi c_{\sigma}(Y,\lambda).
\]
\end{proposition}
\noindent
Here, and in similar limits below, it is understood that the limit is over $r$ such that the chain complex $\widehat{CM}^*$ is defined.

The proof of Proposition~\ref{prop:energyaction} requires two preliminary lemmas which will also be needed later. To state the first lemma, recall from \cite[Prop.\ 2.8]{tw2} that in the case $\Gamma=0$, if $r$ is sufficiently large then there is a unique (up to gauge equivalence) ``trivial'' solution $(A_{triv},\psi_{triv})$ to \eqref{eqn:taubessw} such that $1-|\psi|<1/2$ on all of $Y$. If $(\lambda,J)$ is $L$-flat with $L>0$, then $(A_{triv},\psi_{triv})$ corresponds to the empty set of Reeb orbits under the isomorphism \eqref{eqn:cci} with $\Gamma=0$, see the beginning of \cite[\S3]{e2}.  Any solution not gauge equivalent to $(A_{triv},\psi_{triv})$ will be called ``nontrivial''. Let $L_0$ denote one half the minimum symplectic action of a Reeb orbit.

\begin{lemma}
\label{lem:elb}
There exists an $r$-independent constant $c$ such that if $r$ is sufficiently large, then every nontrivial solution $(A,\psi)$ to \eqref{eqn:taubessw} satisfies $\E(A)>2\pi L_0$ and
\begin{equation}
\label{eqn:elb}
 |cs(A)| \le cr^{2/3}\E(A)^{4/3}.
\end{equation}
\end{lemma}

\begin{proof}
The chain complex $ECC_*^{L_0}(Y,\lambda,\Gamma,J)$ has no generators unless $\Gamma=0$, in which case the only generator is the empty set of Reeb orbits. In particular, the pair $(\lambda,J)$ is $L_0$-flat. By \eqref{eqn:cci}, if $r$ is sufficiently large then every nontrivial solution $(A,\psi)$ to \eqref{eqn:taubessw} has $\E(A)\ge 2\pi L_0$. Given this positive lower bound on the energy, the estimate \eqref{eqn:elb} now follows as in \cite[Eq.\ (4.9)]{tw1}. Note that it is assumed there that $\E(A)\ge 1$, but the same argument works as long as there is a positive lower bound on $\E(A)$.
\end{proof}

Now fix a positive number $\gamma$ such that $\gamma<\delta/4$.

\begin{lemma}
\label{lem:csbound}
For every integer $j$ there exists $\rho \ge 0$ such if $r\ge \rho$ is such that the chain complex $\widehat{CM}^*$ is defined, and if $(A,\psi)$ is a nontrivial irreducible solution to \eqref{eqn:taubessw} of grading $-j$,  then
\begin{equation}
\label{eqn:csbound}
|cs(A)|\le r^{1-\gamma}\E(A).
\end{equation}
\end{lemma}

\begin{proof}
Fix $j$. Let $(A,\psi)$ be a nontrivial solution to \eqref{eqn:taubessw} of grading $-j$ with
 \begin{equation}
\label{eqn:hyp}
|cs(A)|>r^{1-\gamma}\E(A).
\end{equation}
By Lemma~\ref{lem:elb}, if $r$ is sufficiently large then
\begin{equation}
\label{eqn:csn}
|cs(A)| \le cr^{2/3}\E(A)^{4/3}.
\end{equation}
Combining \eqref{eqn:hyp} with \eqref{eqn:csn}, we conclude that $\E(A)\ge c^{-3}r^{1-3\gamma}$. Using \eqref{eqn:hyp} again, it follows that
\[
|cs(A)|> c^{-3}r^{2-4\gamma}.
\]
But this contradicts Proposition~\ref{prop:sfcs} when $r$ is sufficiently large with respect to $j$, since $\delta>4\gamma$.
\end{proof}

\begin{proof}[Proof of Proposition~\ref{prop:energyaction}.]
Choose $L_0>c_\sigma(Y,\lambda)$ and let $(\lambda_1,J_1)$ be an $L_0$-flat approximation to $(\lambda,J)$. For $r$ large,
define $f_1(r)$ to be the infimum over $L$ such that the class $\hat{\sigma}$ is in the image of the map \eqref{eqn:iim1}. We first claim that
\begin{equation}
\label{eqn:limit1}
\lim_{r\to\infty}\left(f_1(r)-c_\sigma(Y,\lambda)\right) = 0.
\end{equation}
This holds because for every $L\le L_0$ which is not the symplectic action of an ECH generator, in particular $L\neq c_\sigma(Y,\lambda)$, if $r$ is sufficiently large that the isomorphism \eqref{eqn:Lflat} is defined, then the class $\hat{\sigma}$ is in the image of the map \eqref{eqn:iim1} if and only if $L>c_\sigma(Y,\lambda)$.

Next define $f(r)$ for $r$ large to be the infimum over $L$ such that the class $\hat{\sigma}$ is in the image of the inclusion-induced map
\begin{equation}
\label{eqn:iim2}
\widehat{HM}^*_L(Y,\frak{s}_{\xi,\Gamma},\lambda,J,r)\to \widehat{HM}^*(Y,\frak{s}_{\xi,\Gamma}).
\end{equation}
It follows from \cite[Lem.\ 3.4(c)]{cc2} that
\begin{equation}
\label{eqn:limit2}
\lim_{r\to\infty}\left(f(r)-f_1(r)\right)=0.
\end{equation}

By \eqref{eqn:limit1} and \eqref{eqn:limit2}, to complete the proof of Proposition~\ref{prop:energyaction} it is enough to show that
\begin{equation}
\label{eqn:limit3}
\lim_{r\to\infty}\left(\E_{\hat{\sigma}}(r) - 2\pi f(r)\right)=0.
\end{equation}

To prepare for the proof of \eqref{eqn:limit3}, assume that $r$ is sufficiently large so that Lemma~\ref{lem:elb} is applicable and Lemma~\ref{lem:csbound} is applicable to $j=-gr(\hat{\sigma})$. Also assume that $r$ is sufficiently large so that all nontrivial Seiberg-Witten solutions in grading $gr(\hat{\sigma})$ are irreducible and have positive energy. Let $(A,\psi)$ be a nontrivial solution in grading $gr(\hat{\sigma})$. Then
\[ 
\mathcal{F}(A,\psi)=\frac{1}{2}(cs(A)-r\E(A))+e_\mu(A).
\]
By \cite[Eq.\ (4.2)]{tw1} and Lemma~\ref{lem:elb}, we have
\begin{equation}
\label{eqn:emubound}
|e_\mu(A)| \le \kappa \E(A)
\end{equation}
where $\kappa$ is an $r$-independent constant.
The above and Lemma~\ref{lem:csbound} imply that
\begin{equation}
\label{eqn:EF1}
(1-r^{-\gamma}-2\kappa r^{-1})\E(A) \le \frac{-2}{r}\mathcal{F}(A,\psi) \le (1+r^{-\gamma}+2\kappa r^{-1})\E(A).
\end{equation}
Also, it follows from the construction of the trivial solution in \cite{tw2} that
\begin{equation}
\label{eqn:EF2} \lim_{r\to\infty}\E(A_{triv})=\lim_{r\to\infty}\frac{\mathcal{F}(A_{triv},\psi_{triv})}{r} = 0.
\end{equation}

Now \eqref{eqn:limit3} can be deduced easily from \eqref{eqn:EF1} and \eqref{eqn:EF2}. The details are as follows. Fix $\epsilon>0$ and suppose that $r$ is sufficiently large as in the above paragraph.
By the definition of $f(r)$, the class $\hat{\sigma}$ is in the image of the map \eqref{eqn:iim2} for $L=f(r)+\epsilon$. Also, if $r$ is sufficiently large, then by \eqref{eqn:EF1} and \eqref{eqn:EF2}, and the fact that $L$ has an upper bound when $r$ is large by \eqref{eqn:limit1} and \eqref{eqn:limit2}, if $\eta=\sum_i(A_i,\psi_i)$ is a cycle in $\widehat{CM}_L$ representing the class $\hat{\sigma}$, then $-2\mathcal{F}(A_i,\psi_i)/r<2\pi (L + \epsilon)$ for each $i$. Consequently $-2\mathcal{F}_{\hat{\sigma}}(r)/r < 2\pi (L+\epsilon)$. By \eqref{eqn:EF1} and \eqref{eqn:EF2} again, if $r$ is sufficiently large then $\E_{\hat{\sigma}}(r) < 2\pi (L+2\epsilon)$, which means that $\E_{\hat{\sigma}}(r)< f(r)+3\epsilon$.

By similar reasoning, if $\E_{\hat{\sigma}}(r)< f(r)-\epsilon$, then if $r$ is sufficiently large, the class $\hat{\sigma}$ is in the image of the map \eqref{eqn:iim2} for $L=f(r)-\epsilon/2$, which contradicts the definition of $f(r)$.
\end{proof}

\subsection{Proof of the upper bound}
\label{sec:proofup}

We are now ready to prove Proposition~\ref{prop:up}. Before giving the details, here is the rough idea of the proof. Consider an ECH class $\sigma$ of grading $j$ large, and choose the absolute gradings so that the corresponding Seiberg-Witten grading is $-j$. Consider a max-min family for the corresponding Seiberg-Witten Floer cohomology class as $r$ increases from $r_j$ to $\infty$. If $(A,\psi)$ is the element of the max-min family for $r=r_j$, then heuristically we have
\[
\frac{4\pi^2c_\sigma(Y,\lambda)^2}{j} \approx \frac{\E(A)^2}{j} \le \frac{\left(\frac{r_j}{2}\op{vol}(Y,\lambda)+C\right)^2}{\frac{1}{16\pi^2}r_j^2\op{vol}(Y,\lambda) + r_j^{2-\delta}} \approx 4\pi^2\op{vol}(Y,\lambda).
\]
The idea of the approximation on the left is that $\E(A)$ converges to $2\pi c_\sigma(Y,\lambda)$ as $r\to\infty$ by Proposition~\ref{prop:energyaction}; and since each member of the max-min family for $r>r_j$ is irreducible by Lemma~\ref{lem:red}, we can apply Proposition~\ref{prop:sfcs} and some calculations to control the change in $\E(A)$ in this family. The inequality in the middle follows from equation
\eqref{eqn:rj} and Lemma~\ref{lem:energybound}. The approximations on the left and right get better as $j$ increases because $\lim_{j\to\infty}r_j = +\infty$.

\begin{proof}[Proof of Proposition~\ref{prop:up}.] The proof has six steps.

{\em Step 1: Setup.\/}
  If $\sigma\in ECH_*(Y,\xi,\Gamma)$ is a nonzero homogeneous class, let $\hat{\sigma}\in \widehat{HM}^*(Y,\frak{s}_{\xi,\Gamma})$ denote the corresponding class in Seiberg-Witten Floer cohomology via the isomorphism \eqref{eqn:taubes}.  We can choose the absolute grading $I$ on $ECH(Y,\xi,\Gamma)$ so that the Seiberg-Witten grading of $\hat{\sigma}$ is  $-I(\sigma)$ for all $\sigma$. For Steps 1--5, fix such a class $\sigma$ and write $j=I(\sigma)$.  We will obtain an upper bound on $c_\sigma(Y,\lambda)$ in terms of $j$ when $j$ is sufficiently large, see \eqref{eqn:step6} below.

To start, we always assume that $j$ is sufficiently large so that $j>0$, the number $r_j$ defined in \eqref{eqn:rj} satisfies $r_j\ge 1$, Proposition~\ref{prop:sfcs} and Lemma~\ref{lem:elb} are applicable to $r\ge r_j$, Lemma~\ref{lem:red} is applicable so that $r_j>s_j$, and the trivial solution $(A_{triv},\psi_{triv})$ does not have grading $-j$.

Fix a max-min family $(A_{\hat{\sigma}}(r),\psi_{\hat{\sigma}}(r))$ for $\hat{\sigma}$ as in \S\ref{sec:maxmin}. Let $S$ denote the discrete set of $r$ for which the chain complex $\widehat{CM}^*$ is not defined. Recall that the max-min family is defined for $r\in(s_j,\infty)\setminus S$. For such $r$, define
\begin{align}
\nonumber
\E(r) &= \E_{\hat{\sigma}}(r)=\E(A_{\hat{\sigma}}(r)),\\
\nonumber
cs(r) &= cs(A_{\hat{\sigma}}(r)),\\
\nonumber
e_{\mu}(r)&=e_\mu(A_{\hat{\sigma}}(r)),\\
\label{eqn:vr}
v(r) &= -\frac{2\mathcal{F}_{\hat{\sigma}}(r)}{r} = \E(r) - \frac{cs(r)}{r}-\frac{2e_\mu(r)}{r}.
\end{align}
It follows from Lemma~\ref{lem:fam} that $v(r)$ extends to a continuous and piecewise smooth function of $r\in(s_j,\infty)$. However the functions $cs(r)$, $\E(r)$, and $e_\mu(r)$ might not extend continuously over $S$. For any equation below involving the latter three functions, it is implicitly assumed that $r\notin S$.

By Lemma~\ref{lem:fam}, we have
\begin{equation}
\label{eqn:vcs}
 \dfrac{dv(r)}{dr}=\frac{cs(r)}{r^2}+\frac{2e_\mu(r)}{r^2}.
\end{equation}
By Proposition~\ref{prop:sfcs} we have the key estimate
\begin{equation}
\label{eqn:jcs}
\left|-j+\frac{1}{4\pi^2}{cs(r)}\right| < Kr^{2-\delta}
\end{equation}
whenever $r\ge r_j$. Here we are using the fact that
$gr(A_{\hat{\sigma}}(r),\psi_{\hat{\sigma}}(r)) = -j$, because
$(A_{\hat{\sigma}}(r),\psi_{\hat{\sigma}}(r))$ is irreducible by Lemma~\ref{lem:red}.

Define a number $\overline{r}=\overline{r}_{\hat{\sigma}}$ as follows.  We know from Lemma~\ref{lem:csbound} that if $r$ is sufficiently large then \begin{equation}
\label{eqn:csboundagain}
|cs(r)| \le r^{1-\gamma}\E(r).
\end{equation}
If \eqref{eqn:csboundagain} holds for all $r\ge r_j$, define $\overline{r}=r_j$. Otherwise define $\overline{r}$ to be the supremum of the set of $r$ for which \eqref{eqn:csboundagain} does not hold.

{\em Step 2.\/}
We now show that
\begin{equation}
\label{eqn:step3}
\limsup_{r\ge \bar{r}}
\E(r)\le v(\bar{r})g(\bar{r}),
\end{equation}
where
\begin{equation}
\label{eqn:g}
g(r)= \exp\left(
\frac{r^{-\gamma}+2\gamma\kappa r^{-1}}{\gamma\left(1-r^{-\gamma}-2\kappa r^{-1}\right)}
\right),
\end{equation}
and $\kappa$ is the constant in \eqref{eqn:emubound}. Here and below we assume that $j$ is sufficiently large so that $1-r_j^{-\gamma}-2\kappa r_j^{-1}>0$.

To prove \eqref{eqn:step3}, assume that $r\ge \bar{r}$.
Then
by \eqref{eqn:vr}, \eqref{eqn:csboundagain}, and \eqref{eqn:emubound}, as in \eqref{eqn:EF1}, we have
\begin{equation}
\label{eqn:energy}
\E(r)\le \frac{1}{1-r^{-\gamma}-2\kappa r^{-1}} v(r).
\end{equation}
Also $v(r)> 0$, since $r\ge 1$.  By \eqref{eqn:vcs}, \eqref{eqn:csboundagain}, \eqref{eqn:emubound} and \eqref{eqn:energy} we have
\[
  \frac{dv(r)}{dr}\le (r^{-1-\gamma}+2\kappa r^{-2})\E(r) \le \frac{r^{-1-\gamma}+2\kappa r^{-2}}{1-r^{-\gamma}-2\kappa r^{-1}}v(r) \le \frac{r^{-1-\gamma}+2\kappa r^{-2}}{1-\bar{r}^{-\gamma} - 2\kappa\bar{r}^{-1}}v(r).
\]
Dividing this inequality by $v(r)$ and integrating from $\bar{r}$ to $r$ gives
\begin{align*}
\ln\left(\frac{v(r)}{v(\bar{r})}\right)
&\le
\frac{\bar{r}^{-\gamma}+2\gamma\kappa \bar{r}^{-1}-r^{-\gamma}-2\gamma\kappa r^{-1}}{\gamma\left(1-\bar{r}^{-\gamma}-2\kappa\bar{r}^{-1}\right)}
\\
&< \frac{\bar{r}^{-\gamma}+2\gamma\kappa\bar{r}^{-1}}{\gamma\left(1 -\bar{r}^{-\gamma}-2\kappa\bar{r}^{-1}\right)}.
\end{align*}
Therefore
\[
 v(r) < v(\bar{r})g(\bar{r}).
\]
Together with \eqref{eqn:energy}, this proves \eqref{eqn:step3}.

{\em Step 3.} We claim now that
\begin{equation}
\label{eqn:step4}
 v(\bar{r})
\le \frac{1}{2}r_j\vol+C_0\bar{r}^{1-\delta}.
\end{equation}
Here and below, $C_0,C_1,C_2\ldots$ denote positive constants which do not depend on $\hat{\sigma}$ or $r$, and which we do not need to know anything more about.

To prove \eqref{eqn:step4}, use  
\eqref{eqn:vcs}, \eqref{eqn:jcs}, \eqref{eqn:emubound}, and Lemma~\ref{lem:energybound} to obtain
\[
\frac{dv}{dr}\le\frac{4\pi^2(j+Kr^{2-\delta})}{r^2}+C_1 r^{-1}.
\]
Integrating this inequality from $r_j$ to $\bar{r}$ and using $j>0$, we deduce that
\begin{equation}
\label{eqn:3}
\begin{split}
 v(\bar{r})-v(r_j)&\le \frac{4\pi^2j}{r_j} - \frac{4\pi^2j}{\bar{r}} + \frac{4\pi^2K(\bar{r}^{1-\delta}-r_j^{1-\delta})}{1-\delta}+C_1(\ln\bar{r}-\ln r_j)\\
& \le \frac{4\pi^2j}{r_j}+C_2\bar{r}^{1-\delta}.
\end{split}
\end{equation}
Also, by \eqref{eqn:vr}, \eqref{eqn:jcs}, \eqref{eqn:emubound}, and Lemma~\ref{lem:energybound}, we have
\begin{equation}
\begin{split}
\label{eqn:4}
 v(r_j)
&\le \frac{1}{2}r_j\vol+C+\frac{4\pi^2(-j+Kr_j^{2-\delta})+2\kappa(r_j
\vol/2+C)}{r_j}\\
&\le\frac{1}{2}r_j\vol-\frac{4\pi^2j}{r_j}+C_3r_j^{1-\delta}.
\end{split}
\end{equation}
Adding \eqref{eqn:3} and \eqref{eqn:4} gives \eqref{eqn:step4}.

{\em Step 4.} We claim now that if $j$ is sufficiently large then
\begin{equation}
\label{eqn:step5}
 \bar{r} \le C_4r_j^{\frac{1}{1-2\gamma}}.
\end{equation}

To prove this, by the definition of $\bar{r}$, if $\bar{r}>r_j$ then there exists a number $r$ slightly smaller than $\bar{r}$ such that $|cs(r)|>r^{1-\gamma}\E(r)$. It then follows from Lemma~\ref{lem:elb} that
\[
r^{1-\gamma}\E(r)< cr^{2/3} \E(r)^{4/3}.
\]
Therefore
\[
 r^{2-4\gamma}\le c^3r^{1-\gamma}\E(r)\le c^3|cs(r)|.
\]
By \eqref{eqn:jcs} and the definition of $r_j$ in \eqref{eqn:rj}, we have
\[
c^3|cs(r)|\le C_5r_j^2+C_6r^{2-\delta}.
\]
Combining the above two inequalities and using the fact that $r$ can be arbitrarily close to $\bar{r}$, we obtain
\[
\bar{r}^{2-4\gamma} \le C_5r_j^2+C_6\bar{r}^{2-\delta}.
\]
Since $\delta>4\gamma$ and $\bar{r}>r_j\to\infty$ as $j\to\infty$, if $j$ is sufficiently large then
\[
C_6\bar{r}^{2-\delta} \le \frac{1}{2}\bar{r}^{2-4\gamma}.
\]
Combining the above two inequalities proves \eqref{eqn:step5}.

Assume henceforth that $j$ is sufficiently large so that \eqref{eqn:step5} holds.

{\em Step 5.} We claim now that
\begin{align}
\label{eqn:step6}
 c_{\sigma}(Y,\lambda)\le \frac{1}{4\pi}r_j\vol g(\bar{r})+C_7r_j^{1-\nu},
\end{align}
where $\nu= 1-\frac{1-\delta}{1-2\gamma}>0$.

To prove \eqref{eqn:step6}, insert \eqref{eqn:step5} into \eqref{eqn:step4} to obtain
\[
 v(\bar{r})\le \frac{1}{2}r_j\vol+C_8r_j^{1-\nu}.
\]
The above inequality and \eqref{eqn:step3} imply that
\[
\begin{split}
\limsup_{r\to\infty}\E(r)&\le \left(\frac{1}{2}r_j\vol + C_8r_j^{1-\nu}\right)g(\bar{r})\\
&\le\frac{1}{2} r_j\vol g(\bar{r})+C_9r_j^{1-\nu}.
\end{split}
\]
It follows from this and Proposition~\ref{prop:energyaction} that \eqref{eqn:step6} holds.

{\em Step 6.}  We now complete the proof of Proposition~\ref{prop:up} by applying \eqref{eqn:step6} to the sequence $\{\sigma_k\}$ and taking the limit as $k\to\infty$.

Let $j_k=I(\sigma_k)$ and $\bar{r}_k=\bar{r}_{\hat{\sigma}_k}$. It then follows from~\eqref{eqn:step6} and the definition of the numbers $r_{j_k}$ in \eqref{eqn:rj} that for every $k$ sufficiently large,
\begin{align}
\label{eqn:final}
\frac{c_{\sigma_k}(Y,\lambda)^2}{I(\sigma_k)}
&\le
\frac{(16\pi^2)^{-1}r_{j_k}^2\vol^2g(\bar{r}_k)^2 + C_{10}r_{j_k}^{2-\nu}}{(16\pi^2)^{-1}r_{j_k}^2\vol-r_{j_k}^{2-\delta}}
\\
\nonumber
&=
\frac{\vol g(\bar{r}_k)^2+C_{11}r_{j_k}^{-\nu}}{1-C_{12}r_{j_k}^{-\delta}}.
\end{align}
By hypothesis, as $k\to\infty$ we have $j_k\to\infty$, and hence $\bar{r}_k>r_{j_k}\to\infty$. It then follows from \eqref{eqn:g} that $\lim_{k\to\infty}g(\bar{r}_k)= 1$. Putting all this into the above inequality proves \eqref{eqn:up}.
\end{proof}

\section{The lower bound}
\label{sec:lb}

In this last section we prove the following proposition, which is the lower bound half of Theorem~\ref{thm:main}:

\begin{proposition}
\label{prop:lb}
Under the assumptions of Theorem~\ref{thm:main},
\begin{equation}
\label{eqn:lb}
\liminf_{k\to\infty} \frac{c_{\sigma_k}(Y,\lambda)^2}{I(\sigma_k)}\ge\vol.
\end{equation}
\end{proposition}

In \S\ref{sec:cob} we review some aspects of ECH cobordism maps, and in \S\ref{sec:provelb} we use these to prove Proposition~\ref{prop:lb}.

\subsection{ECH cobordism maps}
\label{sec:cob}

Let $(Y_+,\lambda_+)$ and $(Y_-,\lambda_-)$ be closed oriented three-manifolds, not necessarily connected, with nondegenerate contact forms. A {\em strong symplectic cobordism\/} ``from'' $(Y_+,\lambda_+)$ ``to'' $(Y_-,\lambda_-)$ is a compact symplectic four-manifold $(X,\omega)$ with boundary $\partial X = Y_+ - Y_-$ such that $\omega|_{Y_\pm} = d\lambda_\pm$. Following \cite{qech}, define a ``weakly exact symplectic cobordism'' from $(Y_+,\lambda_+)$ to $(Y_-,\lambda_-)$ to be a strong symplectic cobordism as above such that $\omega$ is exact.

It is shown in \cite[Thm.\ 2.3]{qech}, by a slight modification of \cite[Thm.\ 1.9]{cc2}, that a weakly exact symplectic cobordism as above induces a map
\[
\Phi^L(X,\omega): ECH^L(Y_+,\lambda_+,0) \longrightarrow ECH^L(Y_-,\lambda_-,0)
\]
for each $L\in\R$, defined by counting solutions to the Seiberg-Witten equations, perturbed using $\omega$, on a ``completion'' of $X$.

More generally, let $A\in H_2(X,\partial X)$, and write $\partial A = \Gamma_+ - \Gamma_-$ where $\Gamma_\pm\in H_1(Y_\pm)$. In our weakly exact symplectic cobordism, suppose that $\omega$ has a primitive on $X$ which agrees with $\lambda_\pm$ on each component of $Y_\pm$ for which the corresponding component of $\Gamma_\pm$ is nonzero.
Then the same argument constructs a map
\begin{equation}
\label{eqn:moregeneral}
\Phi^L(X,\omega,A): ECH^L(Y_+,\lambda_+,\Gamma_+) \longrightarrow ECH^L(Y_-,\lambda_-,\Gamma_-),
\end{equation}
defined by counting solutions to the Seiberg-Witten equations in the spin-c structure corresponding to $A$.
 As in \cite[Thm.\ 2.3(a)]{qech}, there is a well-defined direct limit map
\begin{equation}
\label{eqn:Phi}
\Phi(X,\omega,A) = \lim_{L\to\infty}\Phi^L(X,\omega,A): ECH(Y_+,\xi_+,\Gamma_+) \longrightarrow ECH(Y_-,\xi_-,\Gamma_-),
\end{equation}
where $\xi_\pm=\Ker(\lambda_\pm)$. 

The relevance of the map \eqref{eqn:Phi} for Proposition~\ref{prop:lb}
is that given a class $\sigma_+\in ECH(Y_+,\xi_+,\Gamma_+)$, if
$\sigma_-=\Phi(X,\omega,A)\sigma_+$, then
\begin{equation}
\label{eqn:monotonicity}
c_{\sigma_+}(Y_+,\lambda_+) \ge c_{\sigma_-}(Y_-,\lambda_-).
\end{equation}
The inequality \eqref{eqn:monotonicity} follows directly from \eqref{eqn:Phi} and the definition of $c_{\sigma_\pm}$ in \S\ref{sec:ECH}, cf.\ \cite[Lem.\ 4.2]{qech}.  Here we interpret $c_\sigma=-\infty$ if $\sigma=0$.  By a limiting argument as in \cite[Prop.\ 3.6]{qech}, the inequality \eqref{eqn:monotonicity} also holds if the contact forms $\lambda_\pm$ are allowed to be degenerate.

The map \eqref{eqn:moregeneral} is
a special case of the construction in \cite{field} of maps on ECH
induced by general strong symplectic cobordisms. Without the
assumption on the primitive of $\omega$, these maps can shift the
symplectic action filtration, but the limiting map \eqref{eqn:Phi} is
still defined.

For computations we will need four properties of the map \eqref{eqn:Phi}. First, if $X=([a,b]\times Y,d(e^s\lambda))$ is a trivial cobordism from $(Y,e^b\lambda)$ to $(Y,e^a\lambda)$, where $s$ denotes the $[a,b]$ coordinate, then
\begin{equation}
\label{eqn:trivcob}
\Phi(X,\omega,[a,b]\times\Gamma)=\op{id}_{ECH(Y,\xi,\Gamma)}.
\end{equation}
This follows for example from \cite[Cor.\ 5.8]{cc2}.

Second, suppose that $(X,\omega)$ is the composition of strong symplectic cobordisms $(X_+,\omega_+)$ from $(Y_+,\lambda_+)$ to $(Y_0,\lambda_0)$ and $(X_-,\omega_-)$ from $(Y_0,\lambda_0)$ to $(Y_-,\lambda_-)$. Let $\Gamma_0\in H_1(Y_0)$ and let $A_\pm\in H_2(X_\pm,\partial_\pm X_\pm)$ be classes with $\partial A_+=\Gamma_+-\Gamma_0$ and $\partial A_-=\Gamma_0-\Gamma_-$.
Then
\begin{equation}
\label{eqn:composition}
\Phi(X_-,\omega_-,A_-)\circ\Phi(X_+,\omega_+,A_+) = \sum_{A|_{X_\pm}=A_\pm} \Phi(X,\omega,A).
\end{equation}
This is proved the same way as the composition property in \cite[Thm.\ 1.9]{cc2}.

Third, if $X$ is connected and $Y_\pm$ are both nonempty, then
\begin{equation}
\label{eqn:Ukey}
\Phi(X,\omega,A)\circ U_+ = U_-\circ\Phi(X,\omega,A),
\end{equation}
where $U_\pm$ can be the $U$ map associated to any of the components of $Y_\pm$. This is proved as in \cite[Thm.\ 2.3(d)]{qech}.

Fourth, since we are using coefficients in the field $\Z/2$, it follows from the definitions that the ECH of a disjoint union is given by the tensor product
\begin{equation}
\label{eqn:du3}
ECH((Y,\xi)\sqcup(Y',\xi'),\Gamma\oplus\Gamma') = ECH(Y,\xi,\Gamma)\tensor ECH(Y',\xi',\Gamma').
\end{equation}
If $(X,\omega)$ is a strong symplectic cobordism from $(Y_+,\lambda_+)$ to $(Y_-,\lambda_-)$, and if $(X',\omega')$ is a strong symplectic cobordism from $(Y_+',\lambda_+')$ to $(Y_-',\lambda_-')$, then it follows from the construction of the cobordism map that the disjoint union of the cobordisms induces the tensor product of the cobordism maps:
\begin{equation}
\label{eqn:du4}
\Phi((X,\omega)\sqcup(X',\omega'),A\oplus A') = \Phi(X,\omega,A)\tensor\Phi(X',\omega',A').
\end{equation}

\subsection{Proof of the lower bound}
\label{sec:provelb}

\begin{proof}[Proof of Proposition~\ref{prop:lb}.] The proof has four steps.

{\em Step 1.\/} We can assume without loss of generality that
\begin{equation}
\label{eqn:wlog}
U\sigma_{k+1}=\sigma_k
\end{equation}
for each $k\ge 1$. To see this, note that
  by the isomorphism \eqref{eqn:taubes} of ECH with Seiberg-Witten Floer
  cohomology, together with properties of the latter proved in
  \cite[Lemmas 22.3.3, 33.3.9]{km}, we know that if the grading $*$ is
  sufficiently large, then $ECH_*(Y,\xi,\Gamma)$ is finitely
  generated and
\[
U:ECH_*(Y,\xi,\Gamma)\longrightarrow
  ECH_{*-2}(Y,\xi,\Gamma)
\]
is an isomorphism.  Hence there is a finite collection of sequences
satisfying \eqref{eqn:wlog} such that every nonzero homogeneous class
in $ECH(Y,\xi,\Gamma)$ of sufficiently large grading is contained in
one of these sequences (recall that we are using $\Z/2$ coefficients).
Thus it is enough to prove \eqref{eqn:lb} for a sequence satisfying
\eqref{eqn:wlog}.  Furthermore, in this case \eqref{eqn:lb} is
equivalent to
\begin{equation}
\label{eqn:lowerbound2}
\liminf_{k\to\infty}\frac{c_{\sigma_k}(Y,\lambda)^2}{k} \ge
2\op{vol}(Y,\lambda).
\end{equation}

{\em Step 2.\/} When $(Y,\lambda)$ is the boundary of a Liouville
domain, the lower bound \eqref{eqn:lowerbound2} was proved for a
particular sequence $\{\sigma_k\}$ satisfying \eqref{eqn:wlog} in
\cite[Prop.\ 8.6(a)]{qech}.  We now set up a modified version of this
argument.

  Fix $a>0$ and consider the symplectic manifold
\[
([-a,0]\times
  Y,\omega=d(e^s\lambda))
\]
where $s$ denotes the $[-a,0]$
  coordinate.  The idea is that if $a$ is large, then $([-a,0]\times Y,\omega)$
  is ``almost'' a Liouville domain whose boundary is $(Y,\lambda)$.

Fix $\epsilon>0$.  We adopt the notation that if $r>0$, then
  $B(r)$ denotes the closed ball
 \[
B(r)=\{z\in\C^2\mid \pi|z|^2\le r\},
\]
with the restriction of the standard symplectic form on $\C^2$.
Choose disjoint symplectic embeddings
  \[
\{\varphi_i:B(r_i)\to [-a,0]\times Y\}_{i=1,\ldots,N}
\]
such that $([-a,0]\times Y)\setminus\sqcup_i\varphi_i(B(r_i))$ has symplectic
volume less than $\epsilon$. One can find such embeddings using a covering of $[-a,0]\times Y$ by Darboux charts. Let
\[
X=([-a,0]\times Y)\setminus\bigsqcup_{i=1}^N\op{int}(\varphi_i(B(r_i))).
\]
 Then
  $(X,\omega)$ is a weakly exact symplectic cobordism from
  $(Y,\lambda)$ to $(Y,e^{-a}\lambda)\sqcup\bigsqcup_{i=1}^N\partial
  B(r_i)$. Here we can take the contact form on $\partial B(r_i)$ to be the restriction of the $1$-form $\frac{1}{2}\sum_{k=1}^2(x_kdy_k-y_kdx_k)$ on $\R^4$; we omit this from the notation. Note that there is a canonical isomorphism
\[
H_2(X,\partial X) = H_1(Y).
\]

The symplectic form $\omega$ on $X$ has a primitive $e^s\lambda$ which restricts to the contact forms on the convex boundary $(Y,\lambda)$ and on the component $(Y,e^{-a}\lambda)$ of the concave boundary. Hence, as explained in \S\ref{sec:cob}, we have a well-defined map
\begin{equation}
\label{eqn:PhiX}
\Phi = \Phi(X,\omega,\Gamma): ECH(Y,\xi,\Gamma)\longrightarrow
ECH\left((Y,\xi)\sqcup \bigsqcup_{i=1}^N\partial
  B(r_i),(\Gamma,0,\ldots,0)\right)
\end{equation}
which satisfies \eqref{eqn:monotonicity}.
By \eqref{eqn:du3}, the
target of this map is
\[
ECH\left((Y,\xi)\sqcup \bigsqcup_{i=1}^N\partial
  B(r_i),(\Gamma,0,\ldots,0)\right) =
ECH(Y,\xi,\Gamma)\tensor\bigotimes_{i=1}^N ECH(\partial B(r_i)).
\]
Let $U_0$
denote the $U$ map on the left hand side associated to the component $Y$, and let $U_i$ denote the $U$ map on the left hand side associated to the component $\partial B(r_i)$.
Note that $U_0$ or $U_i$ acts on the right hand side as the tensor product of the $U$ map on the appropriate factor with the identity on the other factors.
By \eqref{eqn:Ukey} we have
\begin{equation}
\label{eqn:U}
\Phi(U_0\sigma) = U_i\Phi(\sigma)
\end{equation}
for all $\sigma\in ECH(Y,\xi,\Gamma)$ and for all $i=0,\ldots,N$.

{\em Step 3.\/} 
 We now give an explicit formula for the cobordism map
$\Phi$ in \eqref{eqn:PhiX}.

Recall that $ECH(\partial B(r_i))$ has a basis $\{\zeta_k\}_{k\ge 0}$ where $\zeta_0=[\emptyset]$ and $U_i\zeta_{k+1}=\zeta_k$.  This follows either from the computation of the Seiberg-Witten Floer homology of $S^3$ in \cite{km}, or from direct calculations in ECH, see \cite[\S4.1]{bn}. We can now state the
formula for $\Phi$:

\begin{lemma}
\label{lem:Phi}
For any class $\sigma\in ECH(Y,\xi,\Gamma)$, we have
\[
\Phi(\sigma) = \sum_{k\ge
  0}\sum_{k_1+\ldots+k_N=k}U_0^k\sigma\tensor\zeta_{k_1}\tensor\cdots\tensor
\zeta_{k_N}.
\]
\end{lemma}

Note that the sum on the right is finite because the map $U_0$
decreases symplectic action.

\begin{proof}[Proof of Lemma~\ref{lem:Phi}.]
Given $\sigma$, we can expand $\Phi(\sigma)$ as
\begin{equation}
\label{eqn:expand}
\Phi(\sigma) = \sum_{k_1,\ldots,k_N\ge
  0}\sigma_{k_1,\ldots,k_N}\tensor\zeta_{k_1}\tensor\cdots\tensor \zeta_{k_N}
\end{equation}
where $\sigma_{k_1,\ldots,k_N}\in ECH(Y,\xi,\Gamma)$.  We need to show
that
\begin{equation}
\label{eqn:c}
\sigma_{k_1,\ldots,k_N}=U_0^{k_1+\cdots+k_N}\sigma.
\end{equation}
We will prove by induction on $k=k_1+\cdots+k_N$ that equation
\eqref{eqn:c} holds for all $\sigma$.

To prove \eqref{eqn:c} when $k=0$, let $X'$ denote the disjoint union of the trivial cobordism $([-a-1,-a]\times Y,d(e^s\lambda))$ and the balls $B(r_i)$.
Then the composition $X'\circ X$ is the trivial cobordism $([-a-1,0]\times Y,d(e^s\lambda))$
from $(Y,e^\lambda)$ to $(Y,e^{-a-1}\lambda)$.  Now each ball $B(r_i)$ induces a cobordism map 
\[
\Phi_{B(r_i)}:ECH(\partial B(r_i))\longrightarrow\Z/2
\]
as in \eqref{eqn:Phi}. By \eqref{eqn:du4} and \eqref{eqn:trivcob} we have
\[
\Phi(X',\Gamma) = \op{id}_{ECH(Y,\xi,\Gamma)}\tensor \Phi_{B(r_1)}\tensor\cdots\tensor\Phi_{B(r_N)}.
\]
It then follows from \eqref{eqn:trivcob} and the composition property \eqref{eqn:composition} that
\[
\begin{split}
\sigma & = (\Phi(X',\Gamma)\circ\Phi)(\sigma)\\
&= \sum_{k_1,\ldots,k_N\ge
  0} \sigma_{k_1,\ldots,k_N} \prod_{i=1}^N\Phi_{B(r_i)}(\zeta_{k_i}).
\end{split}
\]
Now $\Phi_{B(r_i)}$ sends $\zeta_0$ to $1$ by \cite[Thm.\ 2.3(b)]{qech}, and $\zeta_m$ to $0$ for all $m>0$ by grading considerations (the corresponding moduli space of Seiberg-Witten solutions in the completed cobordism has dimension $2m$).  Therefore $\sigma=\sigma_{0,\ldots,0}$ as desired.

Next let $k>0$ and suppose that \eqref{eqn:c} holds for smaller values
of $k$.  To prove \eqref{eqn:c}, we can assume without loss of
generality that $k_1>0$.  Applying $U_1$ to equation
\eqref{eqn:expand} and then using equation \eqref{eqn:U} with $i=1$, we
obtain
\[
\sigma_{k_1,\ldots,k_N} = (U_0\sigma)_{k_1-1,k_2,\ldots,k_N}.
\]
By inductive hypothesis,
\[
(U_0\sigma)_{k_1-1,k_2,\ldots,k_N} = U_0^{k-1}(U_0\sigma).
\]
The above two equations imply \eqref{eqn:c}, completing the proof
of Lemma~\ref{lem:Phi}.
\end{proof}

{\em Step 4.\/} We now complete the proof of
Proposition~\ref{prop:lb}.  Let $\{\sigma_k\}_{k\ge 1}$ be a
sequence in $ECH(Y,\xi,\Gamma)$ satisfying \eqref{eqn:wlog}.  By
\eqref{eqn:monotonicity} we have
\[
c_{\sigma_k}(Y,\lambda) \ge
c_{\Phi(\sigma_k)}\left((Y,e^{-a}\lambda)\sqcup\bigsqcup_{i=1}^N
  \partial B(r_i)\right).
\]
By Lemma~\ref{lem:Phi} and \cite[Eq.\ (5.6)]{qech}, we have
\begin{gather*}
c_{\Phi(\sigma_k)}\left((Y,e^{-a}\lambda)\sqcup\bigsqcup_{i=1}^N
  \partial B(r_i)\right)
=\quad\quad\quad\quad\quad\quad\quad\quad\quad\quad\quad\quad\quad\quad\\
\quad\quad\quad\quad\quad\max_{U^{k'}\sigma_k\neq
  0}\max_{k_1+\cdots+k_N=k'}\left(c_{U_0^{k'}\sigma_k}(Y,e^{-a}\lambda)
  + \sum_{i=1}^Nc_{\zeta_{k_i}}(\partial B(r_i))\right).
\end{gather*}
Since $U^{k-1}\sigma_k=\sigma_1\neq 0$, it follows from the above equation and inequality that
\begin{equation}
\label{eqn:vol1}
c_{\sigma_k}(Y,\lambda)\ge \max_{k_1+\cdots+k_N=k-1}\sum_{i=1}^N
c_{\zeta_{k_i}}(\partial B(r_i)).
\end{equation}

Now recall from \cite{qech} that Theorem~\ref{thm:main} holds for $\partial B(r)$. In detail, we know from \cite[Cor.\ 1.3]{qech} that
\[
c_{\zeta_k}(\partial B(r))=dr
\]
where $d$ is the unique nonnegative integer such that
\[
\frac{d^2+d}{2} \le k \le \frac{d^2+3d}{2}.
\]
Consequently,
\begin{equation}
\label{eqn:vol2}
\lim_{k\to\infty}\frac{c_{\zeta_k}(\partial B(r))^2}{k} = 2r^2= 4\op{vol}(B(r)).
\end{equation}

It follows from \eqref{eqn:vol1} and \eqref{eqn:vol2} and the elementary calculation in \cite[Prop.\
8.4]{qech} that
\begin{equation}
\label{eqn:vol3}
\liminf_{k\to\infty}\frac{c_{\sigma_k}(Y,\lambda)^2}{k} \ge
4\sum_{i=1}^N\op{vol}(B(r_i)).
\end{equation}
By the construction in Step 2,
\begin{equation}
\label{eqn:vol4}
\begin{split}
\sum_{i=1}^N\op{vol}(B(r_i)) & \ge \op{vol}([-a,0]\times
Y,d(e^s\lambda)) - \epsilon\\
&= \frac{1-e^{-a}}{2}\op{vol}(Y,\lambda) - \epsilon.
\end{split}
\end{equation}
Since $a>0$ can be arbitrarily large and $\epsilon>0$ can be arbitrarily small, \eqref{eqn:vol3} and
\eqref{eqn:vol4} imply \eqref{eqn:lowerbound2}. This completes the
proof of Proposition~\ref{prop:lb}.
\end{proof}


\begin{thebibliography}{99}

\bibitem{ch} D. Cristofaro-Gardiner and M. Hutchings, {\em From one Reeb orbit to two\/}, arXiv:1202.4839.

\bibitem{ir} M. Hutchings, {\em The embedded contact homology index
    revisited\/}, New perspectives and challenges in symplectic field
  theory, 263--297, CRM Proc. Lecture Notes 49, Amer. Math. Soc.,
  2009.

\bibitem{tw} M. Hutchings, {\em Taubes's proof of the Weinstein
    conjecture in dimension three\/}, Bull.\ AMS {\bf 47\/} (2010), 73--125.

\bibitem{qech} M. Hutchings, {\em Quantitative embedded contact
    homology\/}, J.\ Diff.\ Geom. {\bf 88\/} (2011), 231--266.

\bibitem{pnas} M. Hutchings, {\em Recent progress on symplectic embedding problems in four dimensions\/}, Proc. Natl. Acad. Sci. USA {\bf 108} (2011), 8093--8099.

\bibitem{bn} M. Hutchings, {\em Lecture notes on embedded contact homology\/}, arXiv:1303.5789, to appear in proceedings of CAST summer school, Budapest, 2012.

\bibitem{field} M. Hutchings, {\em Embedded contact homology as a (symplectic) field theory\/}, in preparation.

\bibitem{obg1} M. Hutchings and C. H. Taubes, {\em Gluing
  pseudoholomorphic curves along branched covered cylinders I\/},
  J. Symplectic Geom. {\bf 5\/} (2007), 43--137.

\bibitem{obg2} M. Hutchings and C. H. Taubes, {\em Gluing
  pseudoholomorphic curves along branched covered cylinders II\/},
  J. Symplectic Geom. {\bf 7\/} (2009), 29--133.

\bibitem{cc2} M. Hutchings and C. H. Taubes, {\em Proof of the Arnold
    chord conjecture in three dimensions II\/}, Geometry and Topology {\bf 17\/} (2013), 2601--2688.

\bibitem{km} P.B. Kronheimer and T.S. Mrowka, {\em Monopoles and
    three-manifolds\/}, Cambridge University Press, 2008.

\bibitem{mcd} D. McDuff, {\em The Hofer conjecture on embedding symplectic ellipsoids\/}, J. Diff. Geom. {\bf 88} (2011), 519--532.

\bibitem{tw1} C. H. Taubes, {\em The Seiberg-Witten equations and the
    Weinstein conjecture\/}, Geom. Topol. {\bf 11\/} (2007),
  2117-2202.

\bibitem{tw2} C. H. Taubes, {\em The Seiberg-Witten equations and the Weinstein conjecture II: More closed integral curves for the Reeb vector field\/}, Geom. Topol. {\bf 13\/} (2009), 1337-1417.

\bibitem{e1} C. H. Taubes, {\em Embedded contact homology and
    Seiberg-Witten Floer cohomology I\/}, Geometry and Topology {\bf
    14\/} (2010), 2497--2581.

\bibitem{e2} C. H. Taubes, {\em Embedded contact homology and
    Seiberg-Witten Floer cohomology II\/}, Geometry and Topology {\bf
   14\/} (2010), 2583--2720.

\bibitem{e3} C. H. Taubes, {\em Embedded contact
  homology and Seiberg-Witten Floer cohomology III\/}, 
Geometry and Topology {\bf
    14\/} (2010), 2721--2817.

\bibitem{e5} C. H. Taubes, {\em Embedded contact homology and
    Seiberg-Witten Floer cohomology V\/}, 
Geometry and Topology {\bf
   14\/} (2010), 2961--3000.

\end{thebibliography}
\end{document}